\documentclass[11pt]{article}
\usepackage{a4,amssymb,amsthm,amsmath,cite,comment,enumerate}
\theoremstyle{plain}
\newtheorem{thm}{Theorem}
\newtheorem{cor}{Corollary}
\newtheorem{lem}{Lemma}
\newtheorem{prop}{Proposition}

\theoremstyle{definition}
\newtheorem{rem}{Remark}

\newtheorem{example}{Example}

\newtheorem{prob}{Problem}
\newtheorem{quest}{Question}

\newtheorem{conj}{Conjecture}

\nonstopmode

\def\cprime{$'$}

\newcommand{\C}{{\mathbb C}} 
\newcommand{\R}{{\mathbb R}} 
\newcommand{\abs}[1]{{\left| {#1} \right|}} \newcommand{\p}[1]{{\left(
      {#1} \right)}}

\renewcommand{\Re}{\operatorname{Re}} 

\renewcommand{\Im}{\operatorname{Im}}

 \newcommand{\Oh}[1]{{O \p{#1}}}

\begin{document}

\author{Johan Andersson\thanks{Department of Mathematics, Stockholm University, SE-106 91 Stockholm SWEDEN, Email: {\texttt {johana@math.su.se}}.}}

\date{}
 
\title{On  the zeta function on the line $\Re(s)=1$.}

\maketitle

\begin{abstract}
  We show  the estimates
  \begin{gather*}\inf_T \int_T^{T+\delta} |\zeta(1+it)|^{-1} dt =\frac{ e^{-\gamma}} 4 \delta^{2}+ \Oh{\delta^4}, \qquad (\delta>0),\\ \intertext{and}
 \inf_T \int_T^{T+\delta} |\zeta(1+it)| dt =\frac{\pi^2 e^{-\gamma}} {24} \delta^{2}+ \Oh{\delta^4}, \qquad (\delta>0),
\end{gather*}
  as   well as corresponding results for sup-norm, $L^p$-norm and other zeta-functions such as the  Dirichlet $L$-functions and certain Rankin-Selberg $L$-functions.
This improves on previous work of Balasubramanian and Ramachandra  for small values of $\delta$ and we remark that it implies that the zeta-function is not universal on the line $\Re(s)=1$.  We also use recent results of Holowinsky (for Maass wave forms) and  Taylor et al. (Sato-Tate for holomorphic cusp forms) to prove lower bounds for the corresponding integral with the Riemann zeta-function replaced with Hecke $L$-functions and  with $\delta^2$ replaced by $\delta^{11/12+\epsilon}$  and $\delta^{8/(3 \pi)+\epsilon}$ respectively.
\end{abstract}

\tableofcontents
\section{Introduction}
\subsection{Classical order and omega estimates}
The importance of studying  the Riemann zeta-function zeta-function on the line $\Re(s)=1$ was first realized by Von Mangoldt who proved in 1895 that $\zeta(1+it) \neq 0$  implies the prime number theorem. Hadamard \cite{Hadamard} and de la Vall{\'e}e-Poussin \cite{Poussin} shortly managed to prove this result independently.  Since then the distribution of the zeta-values on the line $\Re(s)=1$ has been studied by a lot of authors. For the general theory of the Riemann zeta-function, see for example the monographs of Ivi{\'c} \cite{Ivic} and Titchmarsh \cite{Titchmarsh}. Bohr \cite{Bohr} proved that the values $\zeta(1+it)$ are dense in $\C$. Assuming the Riemann hypothesis,  Littlewood \cite{Littlewood} showed that
\begin{gather} \label{A1}
 \zeta(1+it) \ll \log \log t, \qquad  \zeta(1+it)^{-1} \ll \log \log t. \\ \intertext{Bohr and Landau \cite{BohrLandau1,BohrLandau2,BohrLandau3} proved the corresponding omega-estimates} \label{A2}
 \zeta(1+it)=\Omega(\log \log t), \qquad  \zeta(1+it)^{-1}=\Omega(\log \log t),
\end{gather}
unconditionally, so Littlewood's conditional bound should be the best possible. The best unconditional bound are the estimates
\begin{gather*}
 \zeta(1+it) \ll (\log t)^{2/3}, \qquad  \zeta(1+it)^{-1} \ll ( \log t)^{2/3} (\log \log t)^{1/3},
\end{gather*}
of Vinogradov \cite{Vinogradov} and Korobov \cite{Korobov}. For recent improvements and the best known constants in these estimates see the paper of Granville-Soundararajan \cite{Sound}. For related results in this direction, see also Hildebrink \cite{Hildebrink} and Lamzouri \cite{Lamzouri}.
 
\subsection{Universality}
One interesting property for the Riemann zeta-function on lines $\Re(s)=\sigma$ for $1/2<\sigma<1$ is that of universality:

\begin{thm}
 Let $f$ be a continuous function on the interval $[0,H]$ and let $1/2<\sigma<1$. Then for any $\epsilon>0$ there exists a $T$ such that
\begin{gather*} \max_{t \in [0,H]}\abs{f(t)-\zeta(\sigma+iT+it)}<\epsilon.\end{gather*}
\end{thm}
This is a simple consequence of the Universality theorem of Bagchi \cite{Bagchi} (see also Steuding \cite{Steuding} or Laurin{\v{c}}ikas \cite{Laurincikas5}) which generalizes the classical Universality result of Voronin \cite{Voronin0,Voronin2}.  This version of universality  is proved in \cite[Corrolary 2]{Andersson2},  where the requirement that $f(t)$ is nonvanishing on the interval that follows from a trivial application of Bagchi's theorem  is removed.

What about this theorem on the lines $\sigma=1/2$ or $\Re(s)=1$?  A related result of Voronin \cite{Voronin1} that predates his universality result is the following (For a discussion of on how these result are related, see  \cite{Andersson6}):
\begin{thm}
 Suppose that $1/2<\sigma \leq 1.$ Then the set of $n$-tuples 
 \begin{gather*}
   \{ (\zeta(\sigma+it),\zeta'(\sigma+it),\ldots,\zeta^{(n)}(\sigma+it)) : t \in \R \} 
\end{gather*}
 is dense in $\C^n$.
\end{thm}
  For  $\sigma=1/2$ this was recently proved to be false by Garunk{\v{s}}tis and Steuding \cite[Theorem 1. $(ii)$]{GarSteu}. Thus it is not surprising that we manage to show that Theorem 1 is also false on $\sigma=1/2$, although it does not immediately follow from this result. We use a similar argument as that of Garunk{\v{s}}tis and Steuding.  It  follows that the Riemann zeta-function is not universal even in $L^1$ or $L^2$-norm, by the simple fact  that the Hardy $Z$-function is real  and hence the argument of the Riemann zeta-function (up to $\pm 1$) on the critical line is determined by the Gamma-factors of the functional equation. Stirling's formula implies that the Gamma-factors are to regular to allow for universality. For a thorough discussion and a detailed proof of this result see our paper \cite{Andersson6}.

Since Theorem 2 is true for $\sigma=1$ we might guess that Theorem 1 should be true as well for $\sigma=1$. In view of this it might seem surprising that Theorem 1 is infact false on this line.  The proof is  rather simple.  Given large enough $H$  it follows  indirectly from the method in \cite{Andersson} that Theorem 1 can not be valid for $\sigma=1$, since otherwise our proof method which used universality on lines would have worked to disprove a case known to be true.

\subsection{The Balasubramanian Ramachandra method}
More directly, the method of Balasubramanian and Ramachandra gives the following lower estimate \cite{Ramachandra2}
\begin{gather} \label{balram} \int_{T}^{T+H} \abs{\zeta\p{1+it}} dt \geq C_0 H ,  \qquad H \geq H_0,
\end{gather}
for some absolute constants $C_0,H_0>0$ not depending on $T$. It is clear that Theorem 1 is not true (not even in $L^1$-norm)  for $H=H_0$ 
and $\sigma=1$ since if $\delta=C_0 H$, then $\zeta(1+it+iT)$ can clearly not approximate any function $f \in C(0,H)$ with $L^1$ norm less than $\delta$.

The results of Balasubramanian and Ramachandra are very strong and satisfying in some ways. For example their method gives the same omega-estimates for $|\zeta(1+it)|$ and $|\zeta(1+it)|^{-1}$ as \eqref{A2} even in short intervals $t \in [T,T+H]$ if $H=T^\theta$ with $0<\theta<1$. This follows from the result
\begin{gather} \label{A99}
    \max_{T \leq t \leq T+H}|\zeta(1+it)| \gg \log \log H, \qquad \qquad H \geq H_0.
\end{gather}
Assuming the Riemann hypothesis, when $H=T^\theta$ for $0<\theta<1$ this is the best possible result in these intervals up to a constant depending on $\theta$, by Littlewood's result \eqref{A1}.  The possibly exceptions where \eqref{A1} is not true can however be shown to be rather sparse unconditionally. By replacing the Riemann hypothesis with known density theorems \cite[Chapter 11]{Ivic},  we can show that the measure of the exceptional set
 \begin{gather*}
   \operatorname{meas} \, \{0 \leq t \leq T: \text{ Eq. \eqref{A1} is false }  \} \ll_\epsilon T^\epsilon
 \end{gather*} 
 is small for each $\epsilon>0$. This follows from the fact that the set $\{z=\sigma+it: 1-\epsilon/4 \leq \sigma, \, T_0-T_0^{\epsilon/4}  \leq t \leq T_0+T_0^{\epsilon/4} \}$ is a zero free region with at most a measure of $T^\epsilon$ exceptions for $0 \leq T_0<T$. Whenever we have such a zero-free region around $\zeta(1+iT_0)$,  Littlewood's method \cite{Littlewood} applies and the logarithm of the zeta-function can be estimated by Dirichlet polynomial of length some power of $\log T$.

This implies that Balasubramanian-Ramachandra's result \eqref{A99} is the best possible (up to a constant) on average in $T$ even if the Riemann hypothesis is assumed to be false. That is we have that
 \begin{gather} 
    \max_{T \leq t \leq T+H} \abs{\zeta(1+it)}  \ll \log \log T, \qquad \qquad  H \ll T_0^\theta. 
 \end{gather}
for $0 \leq T \leq T_0$, with at most a measure of $T_0^{\theta+\epsilon}$ exceptions.

 That it gives the conjectured right order of magnitude is typical of the method of Balasubramanian and Ramachandra (see Ramachandra's monograph \cite{Ramachandra}). Similarly on the critical line it will give the same lower bound for higher moments of the Riemann zeta-function in short intervals $[T,T+H]$ for $H=T^\theta$ as the conjectured upper bound. In the critical strip it is required that  
\begin{gather*}
 H \gg \log \log T. 
\end{gather*}
in order for Balasubramanian-Ramachandra's method to work, which is weaker than $H \gg 1$, Eq. \eqref{A99}  on the line $\Re(s)=1$. We will discuss the limits of this method in \cite{Andersson4}.

\subsection{A new lower bound in short intervals} 
While it follows from Balasubramanian-Ramachandra's method that the zeta-function is not universal on the line $\Re(s)=1$ for functions $f \in C(0,H_0)$, it does not rule out the existence of some small $0<\delta<H_0$ such that each  continuous function $f(t)$ on the interval $[0,\delta]$ can be approximated by $\zeta(1+iT+it)$. In this paper we will devise new methods that works on $\Re(s)=1$ for arbitrarily short intervals. Our main result will be the following theorem:

\begin{thm}   We have the following estimates for the $L^1$ norm of the zeta-function and its inverse in short intervals: 
   \begin{align*}
     (i)& \qquad   \inf_{T}  \int_T^{T+\delta} \abs{\zeta(1+it)} dt = \frac{\pi^2 e^{-\gamma}}{24} \delta^2+\Oh{\delta^4}, \\
     (ii)& \qquad  \inf_{T}  \int_T^{T+\delta} \abs{\zeta(1+it)}^{-1} dt = \frac{e^{-\gamma}} 4 \delta^2+\Oh{\delta^4}, 
\end{align*}
 for $\delta>0$. Furthermore, both estimates are valid if $\displaystyle \inf_T$ is replaced by $\displaystyle \liminf_{T \to \infty}$.
 \end{thm}
We also have the corresponding theorem when we consider the half-plane $\Re(s)>1$. This will in fact have a slightly simpler proof.
\begin{thm}   We have the following estimates for the $L^1$ norm of the zeta-function and its inverse in short intervals: 
   \begin{align*}
     (i)& \qquad   \inf_{\substack{T \\ \sigma>1}}  \int_T^{T+\delta} \abs{\zeta(\sigma+it)} dt = \frac{\pi^2 e^{-\gamma}}{24} \delta^2+\Oh{\delta^4}, \\
     (ii)& \qquad   \inf_{\substack{T \\ \sigma>1}}   \int_T^{T+\delta} \abs{\zeta(\sigma+it)}^{-1} dt = \frac{e^{-\gamma}} 4 \delta^2+\Oh{\delta^4}, 
\end{align*}
 for $\delta>0$. Furthermore, both estimates are valid if $\displaystyle \inf_T$ is replaced by $\displaystyle \liminf_{T \to \infty}$.
 \end{thm}

\subsection{On which lines is the  zeta-function universal?}
As a consequence of Theorem 3 it is clear that $\zeta(s)$ can not be universal on the line $\Re(s)=1$ even for short intervals. For whenever we have that
\begin{gather*} \int_0^\delta |f(t)| dt< \frac{\pi^2 e^{-\gamma}}{24} \delta^2+\Oh{\delta^4},  \, \, \,  \text{or} \, \, \, \int_0^\delta |f(t)|^{-1} dt<\frac{e^{-\gamma}} 4 \delta^2+\Oh{\delta^4},\end{gather*}
then $\zeta(1+it+iT)$ can not approximate the function $f(t)$ arbitrarily closely even in $L^1$ norm, and certainly not in sup-norm.

By combining Theorem 1 with the observations that the zeta function is not universal on the lines $\Re(s)=1/2$ and $\Re(s)=1$ we obtain the following result:
\begin{thm}
 Under the assumption of the Riemann hypothesis we have that the only lines where the Riemann zeta-function is universal in $L^1$-norm and for some interval $[0,\delta]$ are the lines $\Re(s)=\sigma$ for $1/2<\sigma<1$. Furthermore for those lines we have universality for any interval $[0,H]$ and in sup-norm.
\end{thm}
\begin{proof} 
The only difficult remaining lines  to consider are the  lines $\Re(s)=\sigma$ for $0<\sigma<1/2$. Here we will need the Riemann hypothesis. 
From the Riemann hypothesis it follows that $\log \zeta(\sigma+it) \ll (\log t)^{2-2\sigma+\epsilon}$ \cite[Theorem 14.2]{Titchmarsh}, whenever $1/2 \leq \sigma \leq 1$. From this we see that
\begin{gather} \label{A3}
 \int_{T}^{T+\delta} |\zeta(\sigma+it)| dt \gg_\epsilon  T^{-\epsilon} , \qquad  \qquad (1/2<\sigma<1).
\end{gather}
 By combining \eqref{A3} with the functional equation and Stirling's formula for the Gamma-factors we obtain that
\begin{gather*} 
 \int_{T}^{T+\delta} |\zeta(\sigma+it)|dt \gg_\epsilon  T^{1/2-\sigma-\epsilon}, \qquad   \qquad (0<\sigma<1/2).
\end{gather*}
Since this will tend to infinity as $T \to \infty$ we see that we do not have universality in $L^1$-norm on the line  $\Re(s)=\sigma$ for $0<\sigma<1/2$. That we do not have universality on the lines $\Re(s)=\sigma$ with $\sigma \leq 0$ and the lines that are not parallel to the imaginary axis follows trivially from the functional equation and the definition of the Riemann zeta-function.
\end{proof}

\begin{prob}
 Prove Theorem 5 unconditionally.
\end{prob}

We remark that it would be sufficient to prove eq. \eqref{A3} unconditionally. This however seems quite difficult. In our paper \cite{Andersson6} we use some convexity estimates from Ramachandra's book \cite{Ramachandra} to prove \eqref{A3} under the Lindel\"of hypothesis and thus we have managed to relax the condition of the Riemann hypothesis somewhat.

Results like \eqref{A3} would have other important applications also. 
For example Ivi{\'c} \cite{Ivic2} showed that good lower estimates for this integral (sufficiently explicit in $\delta$) have applications 
on the problem of estimating the multiplicities of the zeta-zeroes.

Possibly another idea can be useful to attack Problem 1?

\section{Some approaches to Theorem 3}

We will first show some approaches to Theorem 3, that although they will not obtain the full strength of Theorem 3,  will yield non trivial results, for example sufficient to prove non universality on the line $\Re(s)=1$. Although the results are superseeded, the ideas might still have some interest. If nothing else they describe how our original approaches to this problem have improved with time. For the reader mainly interested in our final proof this section can be skipped.

\subsection{Lower bound - The Mollifier method}
We will first sketch how to prove  the lower bound in Theorem 3  $(i)$ with $\delta^2$ is replaced by $\delta^{2+\epsilon}$. The lower bound in Theorem 3  $(ii)$ can be proved similarly. This gave us our first of the non universality of the Riemann zeta-function on its abscissa of convergence, and it was first presented at the Zeta Function Days in Seoul, September 1st - 5th, 2009.  We  will find a new use of this method  in \cite{Andersson4}. Introduce the standard Mollifier\footnote{This has also been used by Selberg\cite{Selberg} to show a positive proportion of zeros on the critical line and it has also has important applications for zero density estimates (See \cite{Ivic}, chapter 11).}:
 \begin{gather*}M_X(s)=\sum_{1 \leq n<X} \mu(n) n^{-s}.\end{gather*} 
where  $X = e^{\delta^{-1-\epsilon}}$. Consider the integral
\begin{gather*} 
 (*)=\int_{-\infty}^\infty  \phi((t-T)/\delta) \zeta(1+it) M_X(1+it) dt,
\end{gather*}
for a test function $\phi \in C_0^\infty(\R)$, with support on $[0,\delta]$ such that $2 \pi \hat \phi(0)=1.$ By the triangle inequality we obtain that
\begin{gather*}
 \int_T^{T+\delta} |\zeta(1+it)| dt \gg \frac{(*)}{\max_{t \in [T,T+\delta]} |M_X(1+it)|}.
\end{gather*} 
The lower bound in Theorem 3 $(i)$ with $\delta^2$ replaced by $\delta^{2+\epsilon}$ follows by the estimates 
\begin{gather*}
 |M_X(1+it)| \ll \log X=\delta^{-1-\epsilon}, \qquad \text{and} \qquad 
(*) = \delta+\Oh{\delta^{1+\epsilon}}.
\end{gather*}
Here the first estimate comes from estimating the Dirichlet polynomial  $ M_X(1+it)$ trivially by its absolute values, and the second estimate follows from the fact that $\hat \phi(x) \ll_N x^{-N}$ as $x \to \infty$ for each $N>0$ (Schwartz class maps to Schwartz class under Fourier transforms). By choosing a somewhat smaller value of $X$ and by using theorems of
 Paley and Wiener  (see e.g. \cite[Corrolary 2]{Andersson3}) related to quasianalyticity   on how fast Fourier-transforms of functions with compact support can go to zero, we can obtain improvements of this result. For example it can be shown that $\delta^{2+\epsilon}$ may be replaced by
\begin{gather*}
   \delta^2 \, |\log \delta|^{-1-\epsilon}, \qquad \text{or} \qquad \delta^2 \, |\log \delta|^{-1} \p{\abs{\log |\log \delta|}}^{-1-\epsilon}. 
\end{gather*}
However, the same Paley-Wiener theorem will also give limits for how good estimates this method can yield. For example, this method will not be able to yield the bounds
\begin{gather*}
   \delta^2 \, |\log \delta|^{-1}, \qquad  \text{or} \qquad \delta^2 \, |\log \delta|^{-1} \p{\abs{\log |\log \delta|}}^{-1}. 
\end{gather*}
Hence this method of proof is not strong enough to obtain Theorem 3. 

\subsection{Upper bounds}

It is rather easy to see that the lower estimate in Theorem 3 $(ii)$\footnote{A similar example (although slightly more complicated) can be given for Theorem 3 $(i)$.}, can not be improved to anything better than  something of 
the order of $\delta^2$.  This follows from the example
\begin{example} We have that
\begin{gather*} \int_0^\delta |\zeta(1+it-\delta/2)|^{-1} dt = \frac {\delta^2} 4 + \Oh{\delta^4}.
\end{gather*}
\end{example}
\begin{proof}
 From the  Taylor expansion of the Riemann zeta function at $\Re(s)=1$ it follows that 
\begin{gather*}
 \frac 1 {\zeta(s)} = s-1 - \gamma (s-1)^2 + \Oh{(s-1)^3}.
\end{gather*}
This implies that 
\begin{align*}
  \int_0^\delta |\zeta(1+it-\delta/2)|^{-1} dt &=  \int_0^\delta  
  \abs{(t-\delta/2) + i \gamma (t-\delta/2)^2 + \Oh {(t-\delta/2)^3}} dt \\ &= \frac {\delta^2} 4 + \Oh{\delta^4}.   
 \end{align*}
\end{proof}

One may ask if this counterexample  the best possible? In fact at the Zeta-Function Days in Seoul we asked the following question:
\begin{quest}(Asked at the ZFD in Seoul)
 Suppose that $A(s)$  is a Dirichlet series such that its coefficients and the coefficients of its inverse are absolutely bounded by $1$.
 Is it true that
 \begin{gather*} 
  \int_0^\delta |A(1+it)| dt \geq \int_{-\delta/2}^{\delta/2} |\zeta(1+it)|^{-1} dt,
 \end{gather*}
 with equality iff $A(s)=e^{i \theta} \zeta(s-i \delta/2)^{-1}$?
\end{quest}

 {\em Answer: No}

 Within a month of posing the question we found some other examples that give better estimates:

\begin{example}
If
\begin{gather*}
 A(s)=  \p{\zeta(s)^2 \zeta(s+i \delta/6) \zeta(s-i \delta/6)}^{-1/4}. 
\end{gather*}
Then \begin{gather*}\int_{-\delta/2}^{\delta/2} | A(1+it)| dt \leq 0.9518 \int_{-\delta/2}^{\delta/2} \abs{\zeta(1+it)}^{-1}dt
\end{gather*}
for sufficiently small $\delta$.
\end{example}
This follows from the integral
\begin{gather*}
 \int_{-1/2}^{1/2} \abs{t^2-1/6}^{1/4} \abs{t}^{1/2}dt < \frac {0.9518} 4.
\end{gather*}
These examples are still not as good as what is possible, since
\begin{gather*}
 e^{-\gamma}=0.5615<0.9518<1,
\end{gather*}
they do not give as good results as Theorem 3, which follows from a different construction that gives us something very close to the optimal result.

\section{The logarithmic $L^1$-norm}
\subsection{Jensen's inequality}
We use the following version of Jensen's inequality

\begin{gather} \label{jenseq}
 \frac 1 {\delta} \int_{T}^{T+\delta} \log |\zeta(1+it)| dt \leq  \log  \p{\frac 1 {\delta} \int_T^{T+\delta} |\zeta(1+it)|dt}.
\end{gather}
This version of the Jensen's inequality  can be obtained from the inequality between the arithmetic and geometric means of $N$ points, by letting $N \to \infty$, taking the logarithm of both sides and interpreting the sums as Riemann sums. In general we can replace the logarithm function with any concave function. This inequality has been applied to the zeta-function before, and can be found for example in Titchmarsh \cite{Titchmarsh}, equation 2 on page 230. However it does not seem as anyone applied the inequality on this particular problem before. Theorem 3 thus reduces to the problem of estimating the integral of the logarithm of the zeta-functions in short intervals.

\subsection{The logarithmic $L^1$ norm  of Dirichlet series with multiplicative coefficients}

Jensen's inequality suggests that we should study the logarithm of the Riemann zeta-function. This will in fact be much easier, since we can integrate the logarithm of the Riemann zeta-function term-wise and we have better convergence properties. 

\begin{prop}
 Let $\mathcal M=\{A(s)=\sum_{n=1}^\infty a_n n^{-s}\}$ be the set of Dirichlet series, with completely multiplicative coefficients $a_{nm} =a_n a_m$, such that $|a_n| \leq 1$. Then we have  for $\sigma \geq 1$ that
\begin{align*}
  \inf_{A \in\mathcal  M } \int_{-\delta/2}^{\delta/2} 
  \log \abs{A(\sigma+it)}^{-1} dt, \qquad \text{and} \qquad
  \inf_{A \in  \mathcal M} \int_{-\delta/2}^{\delta/2} 
  \log \abs{A(\sigma+it)} dt
\end{align*}
are minimized by $A(s)=\zeta_{\delta}(s)$, and $A(s)=\zeta(2s)/\zeta_{\delta}(s)$ respectively, where 
\begin{gather*}
  \zeta_\delta(s) = \prod_{p \text{ prime}} (1-\varepsilon_p  p^{-s})^{-1}, \qquad \varepsilon_p=\varepsilon(\delta \log p).
\end{gather*} 
and  $\varepsilon(x)=\operatorname{sign} (\sin \frac x 2) = (-1)^{\lfloor x/(2\pi)  \rfloor}$.
\end{prop}
This  follows from the Euler product and the following Lemma: 
\begin{lem}
 Assume that $0<x<1$, and $\delta>0$.
 Then for $\delta \neq 2 n \pi$ one has that \begin{gather*}
 \int_{-\delta/2}^{\delta/2} \log (1-\varepsilon x e^{i t})dt \leq  \Re\p{\int_{-\delta/2}^{\delta/2} \log (1-x e^{i (\theta+t)})dt} \leq 
  \int_{-\delta/2}^{\delta/2} \log (1+ \varepsilon x e^{i t})dx,
\end{gather*}
 for $\varepsilon=\varepsilon(\delta)=\operatorname{sign} (\sin \frac \delta 2) $. When $\delta=2n\pi$  the inequalities are in fact equalities for any $|\varepsilon|=1$.
\end{lem}

\begin{proof}
 Define 
\begin{gather*}
F(\theta)= \Re \p{\int_{-\delta/2}^{\delta/2} \log (1-x e^{i (\theta+t)})dt}.
\end{gather*}
By the substitution $y=\theta+t$ this can be written as
\begin{gather*}
F(\theta)= \Re \p{\int_{-\delta/2+\theta}^{\delta/2+\theta} \log (1-x e^{i y})dy}.
\end{gather*}  
We will need to determine the extremal values of $F(\theta)$. We do this  as  a calculus excercise. By the fundamental theorem of calculus we get the derivative
\begin{gather*} 
\begin{split}
F'(\theta)&= \Re \p{ \log (1-x e^{i (\theta+\delta/2)})- \log (1-x e^{i (\theta-\delta/2)})}, \\ 
          &= \log \abs{\frac{1-x e^{i (\theta+\delta/2)}}{1-x e^{i (\theta-\delta/2)}}}.
\end{split}
\end{gather*}
This is zero if and only if
\begin{gather*}
(1-xe^{i \theta+i \delta/2})(1-x e^{-i \theta-i \delta/2})  = (1-x e^{i \theta-i \delta/2})(1-x e^{-i \theta+i \delta/2}). \\ \intertext{which simplifies to}
 x  \p{e^{i \theta}-e^{-i \theta}}  \p{e^{i \delta/2}-e^{-i \delta/2}}=0,  \\ \intertext{which is true if}
  \sin \theta=0, \qquad x=0, \qquad \text{or} \qquad \sin (\delta/2)=0.\end{gather*}
We see that $F'(\theta)=0$ if $\delta=2n \pi$ and hence $F(\theta)$ is constant in that case, proving the Lemma in case $\delta=2n  \pi$. Let us now assume that $\delta \neq 2 n \pi$.
Since $x \neq 0$, this means that $F'(\theta)=0$   has the solutions
\begin{gather*} \theta=n \pi.
\end{gather*}
We find  that
\begin{gather*} \begin{split}
 F''(\theta)&= \Re \p{\left[ \frac{-ix  e^{i (\theta+t)}}{1-x   e^{i (\theta+t)}} \right]_{t=-\delta/2}^{t=\delta/2}},  \\ &=
\Re \p{\left[ \frac{-ix  e^{i (\theta+t)}(1-x   e^{-i (\theta+t)})} {|1-x   e^{i (\theta+t)}|^2} \right]_{t=-\delta/2}^{t=\delta/2}}  =
\left[ \frac{x \sin (\theta+t)} {|1-x   e^{i (\theta+t)}|^2} \right]_{t=-\delta/2}^{t=\delta/2}. \end{split}
 \\ \intertext{For  $\theta=2 n \pi$ and $\theta=(2n+1)\pi$ we find that}
F''(2n \pi)=2 \frac{x \sin (\delta/2)} {|1-x   e^{i \delta/2}|^2}, \\ \intertext{and}
F''((2n+1)\pi)=-2 \frac{x \sin (\delta/2)} {|1+x   e^{i \delta/2}|^2},\end{gather*}   
This shows that $\theta=2n\pi$ and $\theta=(2n+1) \pi$ will
be  local maxima or minima for $F(\theta)$ depending on the sign of $\sin (\delta/2)$.
\end{proof}

\noindent {\em Proof of Proposition 1.}

Since its  coefficients are completely multiplicative, the Dirichlet series $A(s)$ and $\zeta_\delta(s)$ have the Euler-products
\begin{gather*} 
 A(s)=\prod_p (1-a_p p^{-s})^{-1},  \qquad \text{and} \qquad \zeta_\delta(s)=\prod_p (1-\varepsilon_p p^{-s})^{-1}.
\\ \intertext{By the fact that  $\varepsilon_p^2=1$  and $1-x^2=(1-x)(1+x)$, we also find that}
 \frac{\zeta(2s)} {\zeta_\delta(s)}=\frac{\prod_p (1-p^{-2s})^{-1}}{\prod_p (1- \varepsilon_p p^{-s})^{-1}}=\frac{\prod_p (1- \varepsilon_p^2  p^{-2s})^{-1}}{\prod_p (1- \varepsilon_p p^{-s})^{-1}} =\prod_p (1+\varepsilon_p p^{-s})^{-1}.
\end{gather*}
By taking the logarithm of these products, we obtain
\begin{gather*}
  \log A(s)=-\sum_p \log (1-a_p p^{-s}), \qquad \log \zeta_\delta(s)=-\sum_p \log (1-\varepsilon_p p^{-s}), \\ \intertext{and} \log \frac{\zeta(2s)}{\zeta_\delta (s)}=-\sum_p \log (1+\varepsilon_p p^{-s}).
\end{gather*}
The Proposition follows by using Lemma 1 termwise.
\qed

As a consequence of Proposition 1, we have the following Lemma:

\begin{lem} We have for $\sigma \geq 1$ that
\begin{align*}
 (i)& & \inf_{T } \int_T^{T+\delta} 
  \log \abs{\zeta(\sigma+it)}^{-1} dt &=  \int_{-\delta/2}^{\delta/2}  \log \abs{\zeta_\delta(\sigma+it)}^{-1} dt, \\
 (ii)& & \inf_{T } \int_T^{T+\delta} 
  \log \abs{\zeta(\sigma+it)} dt &=  \int_{-\delta/2}^{\delta/2}  \log \abs {\frac{\zeta(2\sigma+2it)}{\zeta_\delta(\sigma+it)}} dt,
\end{align*}
 where $\zeta_\delta(s)$ is defined as in Proposition 1.
\end{lem}
\begin{proof}
Since 
\begin{gather*}
  \int_0^\delta \sum_{n=1}^\infty \frac{a_n \Lambda(n)} {n^{\sigma+it} \log n}dt = \sum_{n=1}^\infty \frac{a_n \Lambda(n) (n^{-i \delta}-1)} {-i(\log n)^2 n^\sigma}  
\end{gather*}
is absolutely convergent for any choice of $|a_n| \leq 1$, $\sigma \geq 1$ and $\delta>0$ it follows by using the Euler-product of $\zeta_\delta(s)$ and $\zeta(s)$ and integrating the logarithms term wise that it is sufficient to show that there exist some sequence $T_k$ such that
\begin{gather*}
 \lim_{k \to \infty}\abs{p^{iT_k}-\varepsilon_p}=0, \qquad \text{and}  \lim_{k \to \infty}\abs{p^{iT_k}+\varepsilon_p}=0, \\ \intertext{respectively for each prime $p$ in order for}
 \lim_{k \to \infty}
 \int_{-\delta/2}^{\delta/2} \p{\log \zeta(\sigma+i T_k+it)-\log\zeta_\delta(\sigma+it)}dt=0, \\ \intertext{and}
 \int_{-\delta/2}^{\delta/2} \p{\log \zeta(\sigma+i T_k+it)-\log \frac{\zeta(2\sigma+2it)} {\zeta_\delta(\sigma+it)}}dt=0.
\end{gather*}
But this follows by the fact that $\log p$ are linearly independent over $\mathbb Q$; which is equivalent to the fundamental theorem of arithmetic; and Kroenecker's theorem.
\end{proof}

From this lemma the following analogue of Theorem 3 for the logarithmic $L^1$ -norm  follows.
\begin{thm}
 We have that
 \begin{align*}
(i)& &  \inf_{T } \frac  1 {\delta} \int_T^{T+\delta}  \log \abs{\zeta(1+it)}^{-1} dt &=  \log \delta - \log 4 -\gamma + \Oh{\delta^2}, \\
 (ii)& & \inf_{T } \frac 1 {\delta} \int_T^{T+\delta}  \log \abs{\zeta(1+it)} dt &=    \log \delta -  \log 4 +\log \zeta(2) + \Oh{\delta^2}. \\
\end{align*}
\end{thm}
The lower bound in Theorem 3 will be an  immediate consequence:

\begin{proof} We will sketch a  proof of $(i)$.  Theorem 6. $(ii)$ can be proved by the same method, by replacing the use of Lemma 2 $(i)$ with Lemma 2 $(ii)$. We   give a full proof of  this result by another method later which will yield stronger results (e.g. Proposition 2) also. We use Lemma 2 $(i)$ as a starting point, and hence it is sufficient to calculate
\begin{gather} \label{abaj2}
 \int_{-\delta/2}^{\delta/2}  \log \abs{\zeta_\delta(\sigma+it)} dt  =
\int_{-\delta/2}^{\delta/2}  \log \abs{\zeta(\sigma+it)}dt+
  \int_{-\delta/2}^{\delta/2}  \log \abs{\frac{\zeta_\delta(\sigma+it)}{\zeta(\sigma+it)}} dt,
\end{gather}
for $\sigma>1$.
By using the development of $\log \zeta(1+it)$ as a power series it follows that
\begin{gather} \label{abaj}
 \int_{-\delta/2}^{\delta/2}  \log \abs{\zeta(1+it)} dt =  \delta \log \delta -  (\log 2 +1) \delta + \Oh{\delta^3}.
\end{gather}
By using the Euler products of $\zeta_\delta(s)$ and $\zeta(s)$, taking the logarithms and integrating term wise, we obtain

\begin{gather*}
  \int_{-\delta/2}^{\delta/2}  \log \abs{\frac{\zeta_\delta(\sigma+it)}{\zeta(\sigma+it)}} dt=\sum_{p}  \int_{-\delta/2}^{\delta/2} \log\p{1- \eta\p{\frac{\delta \log p}{4 \pi}}}  p^{-\sigma-it} dt, \\ \intertext{where}  \eta(x)= 
   \begin{cases} 2, & 1/2< \{x \}<1,  \\
    0, & 0 \leq \{x \} \leq   1/2. \end{cases}
\end{gather*}

By using the fact that $\log(1+x)=x+\Oh{x^2}$, and integrating term wise, this equals
\begin{multline*}
  -\sum_{p}  \int_{-\delta/2}^{\delta/2} \eta\p{\frac{\delta \log p}{4 \pi}}  p^{-\sigma-it} dt+\Oh{e^{- 1/\delta}}  =  \\
-4 \sum_{p}    \frac {\sin^-( \delta \log p/2)} {p^\sigma \log p} +\Oh{e^{- 1/\delta}}.
\end{multline*}
where
\begin{gather*}
\sin^-(x)= \begin{cases} 0, & \sin(x) \geq 0, \\ -\sin x, & \sin x <0. \end{cases}
\\  
 \intertext{By letting $\sigma \to 1+$ and using the prime number theorem  this equals}
 -4 \int_0^\infty \frac{ \sin^- (\delta t/2)} {t^2} dt + \Oh{\delta^3}=  -4 \delta \int_0^\infty \frac{ \sin^- ( x/2)} {x^2} dx + \Oh{\delta^3}. \\ \intertext{This integral can be evaluated in terms of Euler's constant and equals}  
 \delta \p{1-\gamma -\log 2}+ \Oh{\delta^3}.
\end{gather*}
The result follows from comining this with \eqref{abaj2} and \eqref{abaj}.
\end{proof}

\vskip 12pt

\noindent{\em Proof of lower bound in Theorem 3.} This follows from Theorem 6 and Jensen's inequality \eqref{jenseq}. \qed

We also remark that our nonuniversality results for the Riemann zeta-function on the line $\Re(s)=1$ can be obtained immediately from Theorem 6, instead of Theorem 3.

Theorem 6 completely solves the problem of obtaining optimal constants in the problem studied in Theorem 3 when the absolute value is replaced by the logarithm of the absolute value  of the function.

\section{Proof of Theorem 3 and Theorem 4}

We have proved that $\zeta_\delta(s)$ is an extremal function for the Logarithmic $L^1$-norm. Hence it is natural to ask what this function will give when we consider the $L^1$-norm of this function in short intervals. Will it give something better than Examples 1 and Examples 2 for answering Question? The surprising answer is that the absolute value of this function will be approximately constant in short intervals, and hence this will essentially give an extremal example also for Question and will yield a proof of Theorem 3 and Theorem 4. This  follows from the following proposition:

 \begin{prop}
 We have uniformly for $-1<x<1$ that
 \begin{gather*}
   \log \zeta_\delta(1+ix \delta/2)= -\log \delta +\gamma +\log 4 - i x \p{\frac 1 {2x} \int_0^{x} \frac { \tan (\pi t/2)} t dt+\gamma} + \Oh{\delta^2}.
 \end{gather*}
\end{prop} 
\begin{rem}
  This shows that the absolute value  $\abs{\zeta_\delta(1+ix \delta/2)}$ is approximately constant when $-1<x<1$, since the imaginary value of the logarithm will not be relevant  when taking absolute values! This is surprising and allows us to prove the upper bounds in Theorems 3 and 4. Further investigation of the function will show that the absolute value of the function will be  approximately constant in the intervals $2n-1<x<2n+1$ for integers $n$, and have discontinuities at  $x=2n+1.$
\end{rem}

We  first prove the following Lemma:

\begin{lem}
 We have that
 \begin{gather*}
   \zeta_\delta(s)= -\log(\delta)+\Theta\p{\frac {s-1} {\delta}} + \log \p{\zeta(s)(s-1)}+ \Oh{(|s|+1)e^{-c_0 \delta^{-1/2}}},  \\ \intertext{where}
   \Theta(s)= \gamma+ \log 4  -\int_0^{s} \frac {\tanh \pi w} w dw,
 \end{gather*}
   uniformly for $\Re(s)>1$ and for some $c_0>0$.
\end{lem}

\begin{rem}
 The error term comes from the zero-free region/error term in the prime number theorem of de la Vall{\'e}e-Poussin \cite{Poussin}. By using the improvements of Vinogradov \cite{Vinogradov}, Korobov \cite{Korobov} and Ford \cite{Ford} this can be improved.  Assuming the Riemann Hypothesis, the error term in Lemma 3 can be improved to $\Oh{(|s|+1)e^{-2\pi/\delta}}$.
\end{rem}

\begin{proof}

We have that

\begin{align*}
 \log \zeta_\delta(s) &= \log \zeta(s)- \sum_{n=1}^\infty \frac{\Lambda(n) \eta\p{\frac{\delta}{4 \pi} \log n }} {n^s \log n} -  \sum_{n \text{ prime power}} \frac{\Lambda(n) \theta\p{\frac{\delta}{4 \pi}  \log n} } {n^s \log n}, \\
&= \log \zeta(s)+ \sum_{n=1}^\infty \frac{\Lambda(n) \eta\p{\frac{\delta\log n}{4 \pi} }} {n^s \log n} +\Oh{e^{-2 \pi/\delta}}, \qquad (\Re(s)\geq 1),
\end{align*}
where
\begin{gather} \label{eta}
 \eta(x)= 
   \begin{cases} 2, & 1/2< \{x\} <1,  \\
    0, & 0 \leq \{x\} \leq   1/2. \end{cases}
\end{gather}
We have that
\begin{gather*}
 (*)=\sum_{n=1}^\infty \frac{\Lambda(n) \eta\p{\frac{\delta}{4 \pi} \log n} } {n^s \log n} =  -\frac 1 {2 \pi i} \int_{c-\infty i}^{c+\infty i} \frac{\zeta'(s+z)}{\zeta(s+z)} 
\int_0^\infty \frac{\theta\p{\frac {\delta}{4 \pi} \log x} x^{z-1} dx} {\log x} dz.
\end{gather*}
By moving the contour to the left of $\Re(z)=0$ we will pick up the zeta-function's residue at $\Re(s+z)=1$. From the zero-free region of  de la Vall{\'e}e-Poussin \cite{Poussin} we get an error term and we find that
\begin{gather*}
 (*)=-\int_0^{\infty} \frac{\eta \p{\frac{\delta \log x}{4 \pi }} x^{-s} dx}{\log x}  +\Oh{(|s|+1)e^{-c_0 \delta^{-1/2}}}. \\ \intertext{By using the substitution } 
 t= \frac{\delta \log x}{4 \pi}, \\ \intertext{we obtain that}
\begin{split}
 (*)&=- \int_0^\infty \frac{\eta(t) e^{-4 \pi (s-1)t/\delta} dt}{t}+ \Oh{(|s|+1)e^{-c_0 \delta^{-1/2}}}, \\  &= -\Psi \p{4 \pi (s-1)/\delta}+\log \p{4 \pi (s-1)/\delta}+ \Oh{(|s|+1)e^{-c_0 \delta^{-1/2}}}, \end{split} \\ \intertext{where} 
  \Psi(s)=-\log s+ \int_0^\infty \frac{\eta(t) e^{-s t}} t dt, 
\end{gather*}
and $\eta(x)$ is defined by \eqref{eta}.
It remains to prove that \begin{gather} \label{tt}
\Theta\p{\frac{s-1} \delta}=\log 4 \pi-\Psi \p{\frac{4 \pi(s-1)}\delta}.
\end{gather}
 It is sufficient to prove that their derivatives coincide and that the value at $\Re(s)=0$ coincide.  We get that\begin{align*}
 \Psi'(s) &=-\int_0^\infty \eta(t) e^{- st} dt+\frac 1 {s} 
            =-2\sum_{n=1}^\infty \int_{n-1/2}^{n} e^{- st} dt +\frac 1 {s} \\
           &= \frac 2 s \sum_{n=1}^\infty  (e^{-ns}-e^{-(n-1/2)s })   +\frac 1 {s} =  \frac 2 s \sum_{k=1}^\infty (-1)^k e^{-ks/2}   +\frac 1 {s} \\
&= \frac 1 s \p{\frac 2 {1+e^{-ks/2}}-1} = \frac 1 {s} \p{\frac 2 {e^{s/2}+1}-\frac {e^{s/2}+1}{e^{s/2}+1}}
 \\ &=\frac  1 {s} \cdot \frac {e^{s/2}-1}{e^{s/2}+1} =  \frac 1 {s} \tanh \frac s 4, 
\end{align*} 
and it is easy to see that \eqref{tt} is true up to a constant, since the derivatives coincide. To determine the constant we  calculate the limit $\lim_{s \to 0^+} \Psi(s)$.
 We have when $0<s \leq  1/2$ that
 \begin{align*}
\Psi(s) &=\int_0^\infty \frac{\eta(x)} t e^{-s t}dt - \log s, \\ &=
  2\sum_{n=1}^\infty \int_{n-1/2}^{n} \frac{e^{-st}} t dt -  \log s, \\ &=
   (2+O(s))  \sum_{n=1}^\infty e^{-sn} \int_{n-1/2}^{n} \frac 1  t dt -  \log s, \\
 &=  (2+O(s))\sum_{n=1}^\infty e^{-sn} \log\p{1+\frac 1{2n-1}}  -  \log s, \\
&= \sum_{n=1}^\infty  \p{2\log\p{1+\frac 1{2n-1}}-\frac 1 n} + 
  \sum_{n=1}^\infty \frac{e^{-ns}}{n} -  \log s + \Oh{s \log s}.  
\end{align*}
 From the explicit evaluation
\begin{gather*}
  \sum_{n=1}^\infty \frac{e^{-ns}}{n}= \log \p{1-e^{-s}},
\end{gather*} 
it  follows that
\begin{gather*}
  \sum_{n=1}^\infty \frac{e^{-ns}}{n} -  \log s =\Oh{s},  \qquad (0 <s \leq 1/2)
\end{gather*} 
and we see that
\begin{gather} \label{ABBA} 
\Psi(s)=
 \sum_{n=1}^\infty  \p{2\log\p{1+\frac 1{2n-1}}-\frac 1 n} + \Oh{s \log s}. \qquad (0<s \leq 1/2)
 \end{gather} 
By the well known estimate
\begin{gather*}
  \sum_{n=1}^N \frac  1 n= \log N+\gamma+\Oh{N^{-1}},
\end{gather*}
we obtain 
\begin{align*}
  \sum_{n=1}^N  \hskip 40pt & \hskip -40 pt \p{2 \log \p{1+\frac 1 {2n-1}}-\frac 1 n} = \\ &= 
2\sum_{n=1}^N \p{\log n - \log \p{n-\frac 1 2}} - \log N-\gamma+\Oh{N^{-1}}, \\ &=
  2\p{\log(\Gamma(N+1)) -\log\p{\frac{\Gamma\p{N+1/2}}{\Gamma(1/2)}}} -\log N  -\gamma+\Oh{N^{-1}}, \\ &= 2 \log \p{\frac{\Gamma(N+1)}{\Gamma\p{N+1/2}}}-\log N+2\log \Gamma(1/2)-\gamma+\Oh{N^{-1}}.
\end{align*}
Stirling's formula
\begin{gather*}
 \log \Gamma(z) = \p{z-\frac 1 2} \log(z)-z+\frac{\log(2 \pi)} 2+ \Oh{z^{-1}},
\end{gather*}
implies that
\begin{gather*} 
2 \log\p{\frac{\Gamma(N+1)}{\Gamma(N+1/2)}}=\log N+\Oh{N^{-1}}.
\end{gather*}
 Thus we have that 
\begin{gather*}
    \sum_{n=1}^\infty \p{2 \log\p{1+\frac 1{2n-1}}-\frac 1 n} = -\gamma+2\log \Gamma(1/2),
\end{gather*}
and from the fact that $\Gamma(1/2)=\sqrt \pi$, together with \eqref{ABBA} we obtain that
\begin{gather*}
 \lim_{s \to 0^+} \Psi(s) = \log \pi -\gamma. 
\end{gather*}
\end{proof}

\subsection{Proof of Proposition 2}
The Laurent series development of the Riemann zeta-function at $s=1$ gives us 
\begin{gather*}
  \log(\zeta(s)(s-1))=- \gamma (s-1)+ \Oh{(s-1)^2}.
\end{gather*}
By Lemma 3 with $s=\sigma +ix\delta/2$ and letting $\sigma \to 1+$,  this estimate gives us
 \begin{gather*}
   \log(\zeta_\delta(1+ix \delta/2))= -\log \delta +\log 4+\gamma -i \gamma x \delta/2-\int_0^{ix/2} \frac{\tanh \pi w} w dw+ \Oh{\delta^2},
\end{gather*}
for $-1<x<1$. With the substitution $t=iw/2$ we obtain
\begin{gather*}
   \log(\zeta_\delta(1+ix))=-\log \delta +\gamma+\log 4 +i\p{-\gamma  x -\frac 1 2 
\int_0^{x} \frac {\tan \pi  t/2}t dt} + \Oh{\delta^2}.
\end{gather*}
\qed

\subsection{Proof of Theorem 3 and 4}

Since the imaginary part of the logarithm can be disregarded (see Remark 2), when taking absolute values, we see that Propostion 2 gives us
\begin{lem}
  We have that uniformly for $-1<x<1$ that
  \begin{align*}
       &(i) & \abs{\frac{\zeta(2(1+ix \delta/2))}{\zeta_\delta(1+ix\delta/2)}}&=\frac{\delta \pi^2 e^{-\gamma}} {24}+ \Oh{\delta^3}, \\
 &(ii)  &   \abs{\zeta_\delta(1+ix \delta)}^{-1} &= \frac{\delta e^{-\gamma}} 4 + \Oh{\delta^3}. \\
  \end{align*}
\end{lem}
\begin{proof}
Part $(ii)$ is an immediate consequence of Proposition 2. Part $(i)$ follows from the  Taylor expansion at $s=2$ for $\zeta(s)$ since we have for real valued $x$ that
\begin{gather*}
\begin{split}
  \abs{\zeta(2+2ix)}&=\abs{\zeta(2)+\zeta'(2)ix+\Oh{x^2}}, \\  &=\sqrt{(\zeta(2)+\Oh{x^2})^2+\Oh{x^2}}, \\ &=\zeta(2)+\Oh{x^2}. \end{split} \end{gather*}
\end{proof}
\noindent {\em Proof of Theorem 4.} The lower bound in Theorem 4 follows from Jensen's inequality and Lemma 2. The upper bound follows in the same way as Lemma 2. Given $\epsilon>0$ and $\sigma>1$ there exists some $T$ such that \begin{gather}\label{tt2} \max_{t \in [-\delta/2,\delta/2]} \abs{\zeta(\sigma+iT+it) - \zeta_{\delta}(\sigma+it)}<\epsilon.\end{gather} By Lemma 4 $(ii)$ and the triangle inequality Theorem 4 $(ii)$ follows. Theorem 4 $(i)$ follows in a similar way. \qed

The proof of Theorem 3 is somewhat more complicated since \eqref{tt2} can not be proved by absolute convergence when $\sigma=1$. We will apply methods coming from the theory of Universality of $L$-functions. We first state a Lemma in slightly more generality than we presently need for later purposes.

\begin{lem}
 Let for some  $\sigma_1<1$ and  $T_0>0$   $$A(s)=\sum_{n=1}^\infty a_n  n^{-s} $$ be a Dirichlet series that is analytic for $\Re(s)>\sigma_1$, $\abs{\Im(s)}>T_0$, absolutely convergent for $\Re(s)>1$  and fullfill the mean square property \begin{gather} \label{msp} \sup_{\sigma>\sigma_1}\int_{T_0}^T \abs{A(\sigma+it)}^2 dt \ll T. \end{gather}  Let $\delta>0$ and $$B(s)=\sum_{n=1}^\infty b_n n^{-s}$$  such that $a_n/b_n$ is a unimodular completely multiplicative function and such that $B(1+it)=\lim_{\sigma \to 1+}  B(\sigma+it)$ is bounded and continuous on $[-\delta/2,\delta/2]$. Then there exists for any $\epsilon>0$ a real number $T$ such that
$$
 \max_{t \in [-\delta/2,\delta/2]}\abs{B(1+it)-A(1+it+iT)}<\epsilon. 
$$
\end{lem}

\begin{proof}
 This follows from the theory of Universality of $L$-functions, and is a variant of e.g.  Steuding \cite[Theorem 4.12]{Steuding} a result due toLaurin{\v{c}}ikas  \cite{Laurincikas1,Laurincikas2}. Note that this is formulated in a slightly different way, and the conditions for the theorem to hold are somewhat different. The same proof method still applies though.
\end{proof}

\noindent {\em Proof of Theorem 3.} 
The lower bound in Theorem 3 follows immediatley from  Lemma 2.  By applying Lemma 5 on $A(s)=\zeta(s)^{-1}$, and $B(s)=\zeta_{\delta+\epsilon}(s)^{-1}$ we find that
$$
  \max_{t \in [-\delta/2,\delta/2]} \abs{\zeta(1+iT+it)^{-1} - \zeta_{\delta+\epsilon}(1+it)^{-1}}<\epsilon. 
$$
Combining this with the triangle inequality and Lemma 4 $(ii)$ this implies that
$$
   \frac{\delta e^{-\gamma}}{4} + \Oh{\delta^3} -\epsilon <\abs{\zeta(1+iT+it)^{-1}}<\frac{\delta e^{-\gamma}}{4} + \Oh{\delta^3} +\epsilon. 
$$
The upper bound in Theorem 3 $(ii)$ then follows  by choosing $0<\epsilon<\delta^3$.
The lower bound follows immediatley from  Lemma 2.

The upper bound in Theorem 3 $(i)$ follows in the same way by Lemma 4 $(ii)$ by choosing $A(s)=\zeta(s)$ and  $B(s)=\zeta(2s)/\zeta_{\delta+\epsilon}(s)$ in Lemma 5. \qed

\subsection{The $L^p$-norm case}

\begin{thm}   We have the following estimates for the $L^p$ norm, for $p>0$ of the zeta-function and its inverse in short intervals: 
   \begin{align*}
     (i)& \qquad   \inf_{T}  \p{\frac 1 {\delta} \int_T^{T+\delta} \abs{\zeta(1+it)}^p dt}^{1/p} = \frac{\pi^2 e^{-\gamma}}{24} \delta+\Oh{\delta^3}, \\
     (ii)& \qquad  \inf_{T}   \p{\frac 1 \delta \int_T^{T+\delta} \abs{\zeta(1+it)}^{-p} dt}^{1/p} = \frac{e^{-\gamma}} 4 \delta+\Oh{\delta^3}, 
\end{align*}
 for $\delta>0$. Furthermore, both estimates are valid if $\displaystyle \inf_T$ is replaced by $\displaystyle \liminf_{T \to \infty}$, and if $1+it$ is replaced by $\sigma+it$ and the infimum is also taken over $\sigma>1$.
\end{thm}

\begin{proof}
 This result follows  in the same way from Proposition 2, Lemma 2, Lemma 4 and Lemma 5 as  Theorem 3 and Theorem  4. \end{proof}
 
\subsection{Sup-norm case}
We can also state our result in sup-norm case. This might in fact be the nicest formulation of our result.

\begin{thm}
  We have that
  \begin{gather*}
    \inf_T \max_{t \in [T,T+\delta]} \abs{\zeta(1+it)}=\frac{e^{-\gamma} \pi^2}{24} \delta + \Oh{\delta^3}, \\ \intertext{and}
      \sup_T \min_{t \in [T,T+\delta]} \abs{\zeta(1+it)} = \frac{4 e^\gamma} \delta + \Oh{\delta}. \end{gather*}
\end{thm}
\begin{proof}
 This result follows  in the same way from Proposition 2, Lemma 2, Lemma 4 and Lemma 5 as  Theorem 3. 
\end{proof}

\section{General Dirichlet series with an Euler product}
We will now show how we can obtain similar results for other Dirichlet series than the Riemann zeta-function. In particular we have Dirichlet $L$-functions and Rankin-Selberg $L$-functions in mind. Hecke $L$-functions of cusp forms will be somewhat more complicated and we will show somewhat weaker results for that case.

\subsection{Multiplicative arithmetical functions}
First we will state a rather general theorem that is valid for Dirichlet series with multiplicative coefficients. Later we will specialise it to the case of completely multiplicative functions and functions with positive coefficients

\begin{thm}
 Let $A(s)$ be a Dirichlet series with multiplicative coefficients
\begin{gather*}
  A(s)=\sum_{n=1}^\infty a(n) n^{-s}  = \prod_{p \text{ prime}} f_p(p^{-s}), \\ \intertext{where}
 f_p(z)= 1+\sum_{k=1}^\infty a(p^k) z^k,
\\ \intertext{such that $A(s)$ is absolutely convergent for $\Re(s)>1$  and}  \begin{split} 
  \sum_{k=2}^\infty   \sum_{\substack{  p \text{ prime } \\ p^k \geq N}} \frac{\abs{a(p^k)}} {p^k} &\ll (\log \log N)^{-2}, \\
  \sum_{p \text{ prime } < N}  \frac{ \abs{a(p)}}{p}   &=\alpha \log \log N + \beta + \Oh{(\log \log N)^{-2}}. \end{split}
\\ \intertext{Then}
   \lambda_1= \sum_{p} \p{\abs{a_p}p^{-1}-\max_{|z|=1/p} \log \abs{f_p(z)}}, \\  \intertext{is convergent and we have for each $0<p<\alpha$ that}
         \inf_{T}   \p{\frac 1 \delta \int_T^{T+\delta} \abs{A(1+it)}^{-p} dt}^{1/p} = \frac{  e^{-\gamma +\beta-\lambda_1   }} {4} \delta^\alpha (1+O(\delta^2)). \\ \intertext{If  furthermore the local Euler-factors $f_p(z)$ have no zeroes for $|z|=1/p$. Then}
  \lambda_0= \sum_{p} \p{\min_{|z|=1/p} \log \abs{f_p(z)}+\abs{a_p}p^{-1}}, \\  \intertext{is convergent and we have for each $0<p<1/\alpha$ that}
   \inf_{T}  \p{\frac 1 {\delta} \int_T^{T+\delta} \abs{A(1+it)}^p dt}^{1/p} =  \frac{  e^{-\gamma +\beta-\lambda_0   }} {4}  \delta^\alpha (1+O(\delta^2)). 
\end{gather*}
Furthermore if the error terms  $(\log \log N)^{-2}$ are  replaced by $o(1)$, the theorem is still true if we replace the error terms $O(\delta^2)$ by $o(1)$.
\end{thm}
\begin{proof} This follows by the same proof method as used to prove Theorem 4.  \end{proof}

\begin{thm}
  Suppose that $A(s)$ fulfill all conditions of Theorem 9 and furthermore that $A(s)$ fulfill the mean square property \eqref{msp} for some $\sigma_1<1$ and is analytic for $\sigma_1 < \Re(s)$, $T_0<\abs{\Im(t)}$.   Then Theorem 9 is true for  all $p>0$. Furthermore we have the corresponding result in sup-norm:
\begin{gather*}
 \inf_T \max_{t \in [T,T+\delta]} \abs{A(1+it)}^{-1} = \frac{  e^{-\gamma +\beta-\lambda_1   }} {4} \delta^\alpha (1+O(\delta^2)). \\ \intertext{If  the local Euler-factors $f_p(z)$ have no zeroes for $|z|=1/p$ then}
\inf_{T}  \max_{t \in [T,T+\delta]}  \abs{A(1+it)}  =  \frac{  e^{-\gamma +\beta-\lambda_0   }} {4}  \delta^\alpha (1+O(\delta^2)). 
\end{gather*}
\end{thm}
\begin{proof}
 The proof uses  Lemma 5 in the same way as the proof of Theorem 3.
\end{proof}

\subsection{Completely multiplicative arithmetical functions}
In the case of completely multiplicative functions we can use the result previously proved, since if $a(n)$ is completely multiplicative we have that
\begin{gather*}
 f_p(z)=\sum_{k=1}^\infty a(p^k) z^k= \sum_{k=1}^\infty (a(p) z)^k = \frac 1 {1-a(p) z}.
\end{gather*}
It is clear that 
\begin{gather*}
  \min_{|z|=1}  \log \abs{f_p(z)}= -\log(1+|a(p)|/p)  \, \, \, \,  \text{and}   \, \, \, \,\max  \log \abs{f_p(z)}= -\log(1-|a(p)|/p).
\end{gather*}
In this case Theorem 9  becomes

\begin{thm}
  Suppose that
\begin{gather*}
  A(s)=\sum_{n=1}^\infty a(n) n^{-s},  \\ \intertext{is a Dirichlet series absolutely convergent for $\Re(s)>1$ such that $a(n)$ is a completely multiplicative function where $|a(n)|=n$ if and only if $n=1$. Suppose that  }
   \sum_{n =1}^N \frac{\Lambda(n)\abs{a(n)}}{n \log n}=\alpha \log \log N+\beta + \Oh{(\log \log N)^{-2}}. \\ \intertext{Then for any $0<p<1/\alpha$}    \begin{split}
   \inf_{T}   \p{\frac 1 \delta \int_T^{T+\delta} \abs{A(1+it)}^{p} dt}^{1/p} &= \sum_{n=1}^\infty \frac{\abs{a(n)}^2}{n^2} \cdot \frac{e^{-\beta}} 4 \delta^\alpha (1+O(\delta^2)), \\ \intertext{and}
    \inf_{T}   \p{\frac 1 \delta \int_T^{T+\delta} \abs{A(1+it)}^{-p} dt}^{1/p} &= \frac{ e^{-\beta }} {4} \delta^\alpha (1+O(\delta^2)). 
\end{split}
\end{gather*}
Furthermore if the error terms  $(\log \log N)^{-2}$ is replaced by $o(1)$, the theorem is still true if we replace the error terms $O(\delta^2)$ by $o(1)$.
\end{thm}
In the completely multiplicative case Theorem 10 becomes
\begin{thm} 
  Suppose that $A(s)$ fulfill all conditions of Theorem 11 and furthermore that $A(s)$ fulfill the mean square property \eqref{msp}  for some $\sigma_1<1$ and is analytic for $\sigma_1 < \Re(s)$, $T_0<\abs{\Im(t)}$.   Then  the results of Theorem 11 are true for any $p>0$. Furthermore we have the corresponding results in sup-norm.
\begin{align*}
 \inf_{T}  \max_{t \in [T,T+\delta]} 
   \abs{A(1+it)} &=\sum_{n=1}^\infty \frac{\abs{a(n)}^2}{n^2} \cdot \frac{e^{-\beta}} {4} \delta^\alpha (1+O(\delta^2)), 
 \\   \sup_{T}    \min_{t \in [T,T+\delta]} \abs{A(1+it)}  &= 
4 e^{\beta } \delta^{-\alpha} (1+O(\delta^2)).  
\end{align*}
\end{thm}


\subsubsection{The Riemann zeta-function revisited}

If $a(n)=1$, we see that $A(s)=\zeta(s)$ is the Riemann zeta-function and  we recover Theorems 3, 4, 7 and 8. 

\subsubsection{Dirichlet L-series}
Since the Dirichlet L-series also have completely multiplicative coefficients we can use Theorem 10, 11 and 12 to obtain a version of Theorem 7 and Theorem 8 (which includes Theorem 3 and 4 as special cases):

\begin{thm}   Let $L(s,\chi)$ be an Dirichlet L-series, where $\chi$ is a character $\mod D$.  Then for $p>0$ we have the following estimates:
   \begin{align*}
     (i)& \qquad   \inf_{T}  \p{\frac 1 {\delta} \int_T^{T+\delta} \abs{L(1+it,\chi)}^p dt}^{1/p} = \frac{ \pi^2 e^{-\gamma}D} {24\Phi(D)} \delta+\Oh{\delta^3},  \\
     (ii)& \qquad  \inf_{T}   \p{\int_T^{T+\delta} \abs{L(1+it,\chi)}^{-p} dt}^{1/p} = \frac{ e^{-\gamma} D} {4 \phi(D)}  \delta+\Oh{\delta^3}, 
\end{align*}
 for $\delta>0$. Furthermore, the results are true if the $L^p$-norm is replaced by the sup-norm and both estimates are valid if $\displaystyle \inf_T$ is replaced by $\displaystyle \liminf_{T \to \infty}$, and if $1+it$ is replaced by $\sigma+it$ and the infimum is also taken over $\sigma>1$.
\end{thm}

\subsection{Functions with positive coefficients}

In the  case of positive coefficients $a(n)$, it will be easier to treat
\begin{gather}
  f_p(z)=\sum_{k=0}^\infty a(p^k) z^k,
\end{gather}
since it is clear that
\begin{gather}
  \max_{|z|=1} \abs{f_p(z)}=f_p(1).
\end{gather}
The minimum is somewhat more complicated. It is not so difficult to see that $z=-1$ is a local minimum of $f_p(z)$ under some minimal assumptions, but it is not a global minimum  in general. Thus we only treat the case with the maximum. 

\begin{thm}
  Suppose that
\begin{gather*}
  A(s)=\sum_{n=1}^\infty a(n) n^{-s}, \tag{a} \\ \intertext{is a Dirichlet series absolutely convergent for $\Re(s)>1$ such that $a(n)$ is a positive multiplicative function, and} \begin{split}
 \sum_{k=2}^\infty   \sum_{\substack {p \text{ prime } \\ p^k \geq N}}  \frac{\abs{a(p^k)}} {p^k} &\ll (\log \log N)^{-2}, \\
 \sum_{n=1}^\infty  \frac{ \Lambda(n) a(n)}{n \log n}   &=\alpha \log \log N + \beta + \Oh{(\log \log N)^{-2}}. \end{split}
\\ \intertext{Then}
   \inf_{T}   \p{\int_T^{T+\delta} \abs{A(1+it)}^{-p} dt}^{1/p} = \frac{e^{-\beta  }} {4} \delta^\alpha (1+O(\delta^2)).
\end{gather*}
Furthermore if the error terms  $(\log \log N)^{-2}$ are replaced by $o(1)$, the theorem is still true if we replace the error term $O(\delta^2)$ by $o(1)$.
\end{thm}

\begin{thm}
  Suppose that $A(s)$ fulfill all conditions of Theorem 14 and furthermore that $A(s)$ fulfill the mean square property \eqref{msp} for some $\sigma_1<1$ and is analytic for $\sigma_1 < \Re(s)$, $T_0<\abs{\Im(t)}$. Then Theorem 14 is true also for $p \geq \alpha$ and with $L^p$-norm replaced by sup-norm.
\end{thm}

\subsubsection{The Rankin-Selberg $L$-function}

Let $f(z)$, and $g(z)$ be automorphic forms on the full modular group,  with Fourier coefficients $a(n)$ and $b(n)$ respectively. The Rankin-Selberg $L$-function is defined by 
\begin{gather*}
 L(s,f \times g) =\zeta(2s) L(s,f \otimes g), \\ \intertext{where}
 L(s,f \otimes g)=\sum_{n=1}^\infty a(n) b(n) n^{-s},
\end{gather*}
denotes the convolution Rankin-Selberg $L$-function. If $g(z)=f(z)$ is a Hecke Eigenform we get that
\begin{gather} \label{ranksel}
  L(s,f \times f)=  \zeta(2s) L(s,f\otimes f)=\zeta(s) \sum_{n=1}^\infty t(n)^2 n^{-s}.
\end{gather}
We have that $t(n)$ is real (see e.g.\cite{Motohashi}) and  hence both the convolution Rankin-Selberg $L$-function as well as the Rankin-Selberg zeta function  have positive coefficients, since $t(n)^2$ will be positive and $\zeta(2s)$ and $\zeta(s)/\zeta(2s)$ have positive coefficients. 
Rankin \cite{Rankin} and Selberg \cite{Selberg2} proved that $L(s,f \times f)$ is a holomorphic function in the neighborhood of $\Re(s)=1$ with the exception of a pole of order $1$ and  residue $1/12 \pi$. The corresponding result for Maass wave-forms can be found in  \cite{Motohashi}. It is also well-known that the mean square property is true for the Rankin-Selberg zeta-function in the critical strip sufficiently close to $\Re(s)=1$.  This means that we can apply Theorem 14 and we obtain the theorems:

\begin{thm}
 Let $L(s,f \times \overline{f})$ be a Rankin-Selberg $L$-function defined by \eqref{ranksel}, and where $f$ is a Hecke-Eigen-form. Then
\begin{gather*}
   \inf_{T}   \p{\frac 1 \delta \int_T^{T+\delta} \abs{L(1+it,f \times \overline{f})}^{-p} dt}^{1/p} = \frac{ 3 \pi  e^{-\gamma    }}  \delta  + \Oh{\delta^3}.
\end{gather*}
We also have that there exist a constant $C$ depending on $f$, but not on $\delta$ which can be calculated by Theorem 10, such that
\begin{gather*}
   \inf_{T}   \p{\frac 1 \delta \int_T^{T+\delta} \abs{L(1+it,f \times \overline{f})}^{p} dt}^{1/p} = C \delta + \Oh{\delta^3}
\end{gather*}
  The corresponding result are also true when the $L^p$-norm is replaced by the  sup-norm.\end{thm}
\begin{proof} The needed results to apply Theorem 10 and Theorem 15 are well-known for the Rankin-Selberg $L$-function. The constant comes from the residue of the Rankin-Selberg $L$-function at $s=1$.  \end{proof} 
\begin{thm}
 Let $L(s,f \otimes \overline{f})$ be a convolution Rankin-Selberg $L$-function defined by \eqref{ranksel}, and where $f$ is a Hecke-Maass cusp form. Then
\begin{gather*}
   \inf_{T}   \p{\frac 1 \delta \int_T^{T+\delta} \abs{L(1+it,f \otimes \overline{f})}^{-p} dt}^{1/p} = \frac{  \pi^3 e^{-\gamma    }} {2} \delta  + \Oh{\delta^3}.
\end{gather*}
We also have that there exist a constant $C$ depending on $f$, but not on $\delta$ which can be calculated by Theorem 10, such that
\begin{gather*}
   \inf_{T}   \p{\frac 1 \delta \int_T^{T+\delta} \abs{L(1+it,f \otimes \overline{f})}^{p} dt}^{1/p} = C \delta + \Oh{\delta^3}
\end{gather*}
\end{thm}
\begin{proof} The needed results to apply Theorem 10 and Theorem 15 are well-known for the convolution $L$-function. It is can also be seen from Theorem 16 and Eq. \eqref{ranksel}. \end{proof}

\begin{prob}
  Calculate the constants C in Theorems 16 and 17.
\end{prob}
The results in this section are proven for $GL(2)$ $L$-functions. However similar results follows for $GL(3)$ and $GL(4)$  $L$-functions as well, by results of Kim \cite{Kim} and Kim and Shahidi \cite{KimShahidi,KimShahidi2}. We will not do this in this paper. 

\subsubsection{Higher order convolution $L$-functions}

Let as in the previous section $a(n)$ be the Fourier coefficients of a $GL(2)$ $L$-function. The higher order convolution $L$-function is defined as follows
\begin{gather} \label{conv}
 L(s,\otimes^{2k} f)= \sum_{n=1}^\infty a(n)^{2k} n^{-s}.
\end{gather}
From the results of Kim-Shahidi\cite{KimShahidi,KimShahidi2} and Kim \cite{Kim} for  the symmetric $n$-th power $L$-functions for $n=1,2,3,4,5,6,7,8$ if follows that
 \begin{gather} \label{aba}
 L(s,\otimes^{2k} f)= \sum_{n=1}^\infty a(n)^{2k} n^{-s} \sim C_k (\log s)^{C(k)} (1+O(\log s)^{-1}), 
\end{gather}
for $k=1,2,3,4$, where \begin{gather}
  C(k)=\frac{(2k)!}{(k+1)! k!}
\end{gather}
denote the Catalan numbers.
Similarly it follows that if $f$ is a holomorphic cusp form that \eqref{aba} is true for all $k \geq 1$, by the recent result of Taylor et. al. \cite{Taylor} who proved that the Symmetric $n$-th power $L$-functions are holomorphic and nonvanishing for $\Re(s) \geq 1$ which already Serre \cite{Serre} had showed implied the Sato-Tate conjecture.

\begin{thm}
 Let $f$ be a Hecke-Maass eigenform for the full modular group. Then there exists constants $A_k$ and $B_k$ for $k=1,2,3,4$, and for any $k \geq 2$ integer if $f$ is a holomorphic form such that for any $1/C(k)>p>0$ we have that
\begin{gather*}
   \inf_{T}   \p{\int_T^{T+\delta} \abs{ L(s,\otimes^k f)}^{-p} dt}^{1/p} = A_k \delta^{C(k)} (1+  \Oh{\delta^2}).
\end{gather*}
We also have that there exist a constant $C$ depending on $f$, but not on $\delta$ which can be calculated by Theorem 9,10, such that
\begin{gather*}
   \inf_{T}   \p{\int_T^{T+\delta} \abs{ L(s,\otimes^k f)}^{p} dt}^{1/p} = B_k\delta^{C(k)} (1+  \Oh{\delta^2}). 
\end{gather*}
Furthermore for $k=1,2,3,4$ we can choose any $p>0$ and the corresponding result for sup-norm is also true.
\end{thm}
\begin{proof} The result follows from Theorems 9,10 and 14,15 and by \eqref{aba} and some Tauberian argument. For the Maass-Wave-form case, it follows from Kim-Shahidi's result that the required mean-square property is true. \end{proof}

\section{A lower bound for more general Dirichlet series}

We remark that what we needed in the proof of the lower bound in Theorems 3 and 4 was the following estimate
\begin{gather*}
 \sum_{n=1}^N \frac{\Lambda(n)} {n \log n}= \log \log N + O(1), 
\end{gather*}
as well as Jensen's inequality. Hence the prime number theorem suffices. Similarly the same proof method can be used to estimate Dirichlet series when similar estimates for the coefficients of the logarithm of the Dirichlet series applies.
\begin{thm} If 
\begin{gather*}
 \log A(s)  = \sum_{n=1}^\infty a_n n^{-s}, 
\end{gather*}
and
\begin{gather*}
 \sum_{n=1}^N \frac{|a_n|} {n} \leq  \alpha \log \log N + O(1),  \qquad  (0<\alpha< 1). 
\end{gather*}
Then
\begin{gather*} \delta^{1+\alpha} \ll \int_T^{T+\delta} \abs{A(1+it)} dt \ll \delta^{1-\alpha}, \qquad 0 <\delta \leq 1.
\end{gather*}
\end{thm}
\begin{proof}
 It follows that 
$$-\alpha \log \delta +O(1) \leq   \frac 1 \delta \int_T^{T+\delta} \log \abs{A(1+it)} dt \leq \alpha \log \delta + O(1)
$$
in a similar way as the proof of  Theorem 6. The conclusion follows from Jensen's inequality.
\end{proof}


\subsection{$L^p$ estimates of the Riemann zeta-function}
Although Theorem 3 gives a stronger result for the lower bound we obtain the following conseqence of Theorem 19 when applied on  the Riemann zeta-function:
\begin{cor}
\begin{gather*}\delta^{|p|} \ll \frac 1 {\delta} \int_T^{T+\delta} \abs{\zeta(1+it)}^{p} dt \ll \delta^{-|p|} \qquad (0<|p|,\delta<1)\end{gather*}
\end{cor}
We remark that the upper bound shows that $\zeta^\theta(s)$ for $|\theta|<1/2$ belongs to the Hardy class $H^2$ of Dirichlet series, see \cite[remark 4]{Andersson111}.

\subsection{Maass wave forms and a result of Holowinsky}

We have the following results for Hecke L-series of Hecke-Maass cusp forms:
\begin{cor}  Let $H(s)$ be a Hecke L-series attached to a Maass 
wave form or holomprphic cusp form. Then
\begin{gather*} \delta^{23/12} \ll \int_T^{T+\delta} \abs{H(1+it)}dt \ll \delta^{1/12}, \qquad (0<\delta \leq 1)   \end{gather*}
\end{cor}
\begin{proof}
This follows from a recent results of Holowinsky. Let $$H(s)=\sum_{n=1}^\infty a(n) n^{-s}.$$ From Lemma 4.1 in  \cite{Holowinsky} it follows that
\begin{gather*}
 \sum_{p \leq X} \frac{\abs{a(p)}}{p} \leq 11/12 \log \log X +\Oh{1}.
\end{gather*}
Corollary 2 is now a consequence of Theorem 19.
\end{proof}
 Holowinsky's result was proved by using recent important results of Symmetric $n$-th power $L$-functions  of Kim and Shahidi \cite{KimShahidi} for $n=1,\ldots,8$ (Just the cases $n=2,4,6$ are in fact used by Holowinsky),  and  was  an important ingredient in his and Soundararajan's proof of Quantum unique ergodicity \cite{HolSound}. We remark that a previous result of Elliot-Moreno-Shahidi \cite{Elliott} proves this with the somewhat weaker  bounds $\delta^{35/18}$ and $\delta^{1/18}$ under the assumption of the Ramanujan-Petersson conjecture, and hence by Deligne \cite{Deligne} for holomorphic cusp forms. Also in the case of Holomorphic cusp forms improvements along these lines have been done by Tenenbaum \cite{Tenenbaum} and Wu \cite{Wu}, although these results are superseeded by the latest results on Sato-Tate of Taylor et.al.

\subsection{Holomorphic cusp forms and the Sato-Tate conjecture}

In the case of holomorphic cusp forms the results of Holowinsky can be improved.  In their recent work  Tom Barnet-Lamb, David Geraghty,  Michael Harris and Richard Taylor \cite{Taylor}   prove the Sato-Tate conjecture for all holomorphic newforms of weight $k \geq 2$ on the group $\Gamma_0(N)$.  By the Shimura-Taniyama conjecture, proved By Wiles \cite{Wiles} for square-free $q$ and completely settled by Breuil-Conrad-Diamond-Taylor  \cite{TaylorST} an $L$-function of an elliptic curve is a cusp newform of weight 2 for $\Gamma_0(q)$. Thus this result
 properly extend some of their previous results on the Sato-Tate for  $L$-functions associated with Elliptic curves with non integral $j$-invariant  \cite{Taylor2,Taylor3,Taylor4}. 
 Let us now assume that
\begin{gather} \label{cuspform} f(z)=\sum_{n=1}^\infty a(n) n^{(k-1)/2}e(nz) \end{gather}
is a holomorphic new form for $\Gamma_0(N)$ of weight $k \geq 2$. 
We remark that in particular this includes all holomorphic cusp forms for the full modular group since all these forms have weight $k \geq 12$. The Sato-Tate conjecture, now a theorem states that $a(p)/2$ should be equidistributed in $[-1,1]$ with respect to the measure $2/\pi \sqrt{1-t^2}$. We obtain that (This is essentially the calculation in \cite[p.511]{Elliott}):
\begin{gather*}
 \sum_{p \leq N} \frac{\abs{a(p)}}{p} \sim 2 \times \frac 2 {\pi} \int_{-1}^1 |t| \sqrt{1-t^2} dt  \times \log \log N \sim \frac 8 {3\pi} \log \log N. 
\end{gather*}
This result and Theorem 19 allows us to improve Corollary 2 in the case of holomorphic cusp forms. Let
\begin{gather} \label{ajj}
   H(s)= \sum_{n=1}^\infty a(n) n^{-s}=\prod_{p}(1-a(p) p^{-s}+p^{-2s})^{-1}
\end{gather}
be the Hecke L-series that corresponds to the cusp form $f(z)$ defined by Eq. \eqref{cuspform}. Then we have the following Corollary

\begin{cor} 
 Let  $H(s)$ be the Hecke L-series attached to a holomorphic new-form of weight $\geq 2$ (Defined by Eqs. \eqref{cuspform} and \eqref{ajj}). Then
\begin{gather}
  \delta^{1+8/(3 \pi)+\epsilon} \ll_\epsilon \int_T^{T+\delta} \abs{H(1+it)}dt \ll_\epsilon  \delta^{1-8/(3 \pi)-\epsilon}, \qquad (0<\delta \leq 1)   
\end{gather}
Furthermore
 \begin{gather*}
  \liminf_{T \to \infty} \int_T^{T+\delta} \abs{H(1+it)}dt \ll_{\epsilon} \delta^{1+8/(3 \pi)-\epsilon}.
 \end{gather*}
\end{cor}
We remark that $8/(3 \pi)=0.848826...$. Elliot-Moreno-Shahidi \cite{Elliott} proves  related results under a somewhat sharper variant of the Sato-Tate conjecture that would allow us to remove the $\epsilon$ in Corollary 3. What they proved under this somewhat sharper version of Sato-Tate is the following
\begin{gather} \label{expst}
 \sum_{p \leq N} \frac{\abs{a_p}}{p} = \frac 8 {3\pi} \log \log N + O(1). 
\end{gather}
 This would follow if we would have a sufficiently good estimates (explicit in $n$) for the symmetric $n$'th power $L$-functions close to $\Re(s)=1$.  
We remark that a conjecture of   Akiyama-Tanigawa \cite{AkiTan} proposed for elliptic curves $L$-functions (but it should be possibly to state for other $L$-functions as well)  would also imply this, and in fact would under a general version of the Riemann hypothesis yield the much stronger error term $\Oh{N^{-1/2+\epsilon}}$ in Eq. \eqref{expst}.  For a discussion about this, see the survey of Mazur \cite{Mazur}.  We therefore suggest the following problem:
\begin{prob}
  Prove that it is possible to remove $\epsilon$ in Corollary 3.
\end{prob}
The error term $O(\log \log N)^{-2}$ in \eqref{expst} would by Theorem 10 yield even sharper results. In fact we believe the following to be true (in particular it would follow from Akiyama-Tanigawa's conjecture):
\begin{conj}
 Let $a_n$ be the Fourier coefficients for a modular form. Then
\begin{gather*} 
 \sum_{p \leq N} \frac{\abs{a_p}}{p} = \frac 8 {3\pi} \log \log N + B+ O(\log \log N)^{-2}. 
\end{gather*}
 for some constant B (depending on the form).
\end{conj}
Then it follows that
\begin{thm}
  Let $a(n)$ be the Fourier coefficients for a modular form such that Conjecture is true. Then there exists some constants $C_1,C_2$ such that
  \begin{align*}
   \inf_T \max_{t \in [T,T+\delta]}  \abs{H(1+it)}  &= C_1 \delta^{8/(3 \pi)} (1+O(\delta^2)), \\
   \sup_T \min_{t \in [T,T+\delta]}  \abs{H(1+it)}  &= C_2 \delta^{-8/(3 \pi)} (1+O(\delta^2)). 
  \end{align*}
   The corresponding result for $L^p$ norm would also be true.
\end{thm}
\begin{proof} This follows from Theorem 10. \end{proof}

\section{Hilbert modular forms}
Recently Barnet-Lamb, Geraghty and Toby  \cite{BarneLamb} proved the Sato-Tate conjecture  for Hilbert modular forms by the same potential automorphy argument used for classical modular forms. Although we are not stating them here, by similar reasoning as in this paper analogues for Theorems 18, Theorem 20 and Corollary 3 can be found for Hecke L-series of Hilbert modular forms.

\def\cprime{$'$} \def\polhk#1{\setbox0=\hbox{#1}{\ooalign{\hidewidth
  \lower1.5ex\hbox{`}\hidewidth\crcr\unhbox0}}}


\begin{thebibliography}{10}

\bibitem{AkiTan}
S.~Akiyama and Y.~Tanigawa.
\newblock Calculation of values of {$L$}-functions associated to elliptic
  curves.
\newblock {\em Math. Comp.}, 68(227):1201--1231, 1999.

\bibitem{Andersson}
J.~Andersson.
\newblock Disproof of some conjectures of {K}. {R}amachandra.
\newblock {\em Hardy-Ramanujan J.}, 22:2--7, 1999.

\bibitem{Andersson2}
J.~Andersson.
\newblock Lavren{\cprime}ev's approximation theorem with nonvanishing
  polynomials and universality of zeta-functions.
\newblock {\em New Directions in Value-distribution Theory of Zeta and $L$-functions: Wurzburg Conference, October 6-10, 2008 (Berichte aus der Mathematik)}, December 31, 2009, pages 7-10.



\bibitem{Andersson6}
J.~Andersson.
\newblock Non universality on the critical line.
\newblock  Forthcoming.

\bibitem{Andersson4}
J.~Andersson.
\newblock On the Balasubramanian-Ramachandra method close to $\Re(s)=1$.
\newblock Forthcoming.

\bibitem{Andersson3}
J.~Andersson.
\newblock On a problem of {R}amachandra and approximation of functions by
  {D}irichlet polynomials with bounded coefficients.
\newblock Forthcoming.

\bibitem{Andersson111}
J.~Andersson.
\newblock On generalized Hardy classes of Dirichlet series.
\newblock Forthcoming.

\bibitem{Bagchi}
B.~Bagchi.
\newblock A joint universality theorem for {D}irichlet {$L$}-functions.
\newblock {\em Math. Z.}, 181(3):319--334, 1982.

\bibitem{Taylor}
T.~Barnet-Lamb, D.~Geraghty, M.~Harris, and R.~Taylor.
\newblock A family of {C}alabi-{Y}au varieties and potential automorphy II.
\newblock {\em P.R.I.M.S.} 47 (2011), 29-98. 



\bibitem{BarneLamb}
T.~Barnet-Lamb, D.~Geraghty, and G.~Toby.
\newblock The {S}ato-{T}ate conjecture for {H}ilbert modular forms
\newblock  47 (2011), 29-98. 
{\em Journal of the AMS}, 24 (2011), 411-469

\bibitem{Bohr}
H.~Bohr.
\newblock \"{U}ber eine quasi-periodische {E}igenschaft {D}irichletscher
  {R}eihen mit {A}nwendung auf die {D}irichletschen {$L$}-{F}unktionen.
\newblock {\em Math. Ann.}, 85(1):115--122, 1922.

\bibitem{BohrLandau1}
H.~Bohr and E.~Landau.
\newblock \"uber das {V}erhalten von $\zeta(s)$ and $\zeta_{\mathfrak r}(s)$ in
  der {N}\"ahe der {G}eraden $\sigma=1$.
\newblock {\em G\"ottinger Nachrichten}, pages 303--330, 1910.

\bibitem{BohrLandau2}
H.~Bohr and E.~Landau.
\newblock \"uber das {V}erhalten von $1/\zeta(s)$ auf der {G}eraden $\sigma=1$.
\newblock {\em G\"ottinger Nachrichten}, pages 71--80, 1923.

\bibitem{BohrLandau3}
H.~Bohr and E.~Landau.
\newblock Nachtrag zu unseren {A}bhandlungen aus den {J}ahrg\"angen 1910 und  1923.
\newblock {\em G\"ottinger Nachrichten}, pages 168--172, 1924.

\bibitem{TaylorST}
C.~Breuil, B.~Conrad, F.~Diamond, and R.~Taylor.
\newblock On the modularity of elliptic curves over {$\mathbf Q$}: wild 3-adic  exercises.
\newblock {\em J. Amer. Math. Soc.}, 14(4):843--939 (electronic), 2001.

\bibitem{Taylor3}
L.~Clozel, M.~Harris, and R.~Taylor.
\newblock Automorphy for some {$l$}-adic lifts of automorphic mod {$l$}
  {G}alois representations.
\newblock {\em Publ. Math. Inst. Hautes \'Etudes Sci.}, (108):1--181, 2008.
\newblock With Appendix A, summarizing unpublished work of Russ Mann, and
  Appendix B by Marie-France Vign{\'e}ras.

\bibitem{Poussin}
C.-J. de~la Vallée~Poussin.
\newblock Recherches analytiques la théorie des nombres premiers.
\newblock {\em Ann. Soc. scient. Bruxelles}, 20:183--256, 1896.

\bibitem{Deligne}
P.~Deligne.
\newblock La conjecture de {W}eil. {I}.
\newblock {\em Inst. Hautes \'Etudes Sci. Publ. Math.}, (43):273--307, 1974.

\bibitem{Elliott}
P.~D. T.~A. Elliott, C.~J. Moreno, and F.~Shahidi.
\newblock On the absolute value of {R}amanujan's {$\tau $}-function.
\newblock {\em Math. Ann.}, 266(4):507--511, 1984.


\bibitem{Ford}
K.~Ford.
\newblock Vinogradov's integral and bounds for the {R}iemann zeta function.
\newblock {\em Proc. London Math. Soc. (3)}, 85(3):565--633, 2002.



\bibitem{GarSteu}
R.~Garunk{\v{s}}tis and J.~Steuding.
\newblock On the roots of the equation $\zeta(s)=a$.
\newblock arXiv:1011.5339 [math.NT]



 
\bibitem{Sound}
A.~Granville and K.~Soundararajan.
\newblock Extreme values of {$\vert \zeta(1+it)\vert $}.
\newblock In {\em The {R}iemann zeta function and related themes: papers in
  honour of {P}rofessor {K}. {R}amachandra}, volume~2 of {\em Ramanujan Math.
  Soc. Lect. Notes Ser.}, pages 65--80. Ramanujan Math. Soc., Mysore, 2006.

\bibitem{Hadamard}
J.~Hadamard.
\newblock Sur la distribution des z\'eros de la fonction {$\zeta(s)$} et ses
  cons\'equences arithm\'etiques.
\newblock {\em Bull. Soc. Math. France}, 24:199--220, 1896.

\bibitem{Hildebrink}
T.~Hildebrink.
\newblock An arithmetical mapping and applications to $\Omega$-results for the Riemann zeta-function.
\newblock {\em Acta. Arith.}, 139.4, 341--367, 2009.

\bibitem{Taylor4}
M.~Harris, N.~Shepherd-Barron, and R.~Taylor.
\newblock A family of {C}alabi-{Y}au varieties and potential automorphy.
\newblock {\em Ann. of Math. (2)}, 171(2):779--813, 2010.


\bibitem{Holowinsky}
R.~Holowinsky.
\newblock A sieve method for shifted convolution sums.
\newblock {\em Duke Math. J.}, 146(3):401--448, 2009.

\bibitem{HolSound}
R.~Holowinsky and K.~Soundararajan.
\newblock Mass equidistribution for {H}ecke eigenforms.
\newblock {\em Ann. of Math. (2)}, 172(2):1517--1528, 2010.

\bibitem{Ivic2}
A.~Ivi{\'c}.
\newblock On the multiplicity of zeros of the zeta-function.
\newblock {\em Bull. Cl. Sci. Math. Nat. Sci. Math.}, (24):119--132, 1999.

\bibitem{Ivic}
A.~Ivi{\'c}.
\newblock {\em The {R}iemann zeta-function}.
\newblock Dover Publications Inc., Mineola, NY, 2003.
\newblock Theory and applications, Reprint of the 1985 original [Wiley, New
  York; MR0792089 (87d:11062)].

\bibitem{Kim}
H.~H. Kim, 
\newblock Functoriality for the exterior square of ${\rm GL}_4$GL4 and the symmetric fourth of ${\rm GL}_2$GL2. With appendix 1 by Dinakar Ramakrishnan and appendix 2 by Kim and Peter Sarnak.
\newblock {\em J. Amer. Math. Soc.} 16 (2003), no. 1, 139-–183 (electronic). 

\bibitem{KimShahidi}
Henry~H. Kim and Freydoon Shahidi.
\newblock Cuspidality of symmetric powers with applications.
\newblock {\em Duke Math. J.}, 112(1):177--197, 2002.

\bibitem{KimShahidi2}
Henry~H. Kim and Freydoon Shahidi.
\newblock Functorial products for {${\rm GL}_2\times{\rm GL}_3$} and the
  symmetric cube for {${\rm GL}_2$}.
\newblock {\em Ann. of Math. (2)}, 155(3):837--893, 2002.
\newblock With an appendix by Colin J. Bushnell and Guy Henniart.





\bibitem{Korobov}
N.~M. Korobov.
\newblock Estimates of trigonometric sums and their applications.
\newblock {\em Uspehi Mat. Nauk}, 13(4 (82)):185--192, 1958.

\bibitem{Lamzouri}
Y.~Lamzouri
\newblock The two dimensional distribution of values of $\zeta(1+it)$.
\newblock{ \em International Mathematics Research Notices IMRN (2008)} Vol. 2008, article ID rnn106, 48 pp. 

\bibitem{Laurincikas5}
A.~Laurin{\v{c}}ikas.
\newblock {\em Limit theorems for the {R}iemann zeta-function}, volume 352 of
  {\em Mathematics and its Applications}.
\newblock Kluwer Academic Publishers Group, Dordrecht, 1996.


\bibitem{Laurincikas1}
A.~Laurin{\v{c}}ikas.
\newblock  On the limit distribution of the Matsumoto zeta function {I}{I} 
\newblock  {\em Lith. Math. J.} 37 (1996), 371-387


\bibitem{Laurincikas2}
A.~Laurin{\v{c}}ikas.
\newblock A limit theorem in the theory of finite Abelian groups,
\newblock {\em Publ. Math. Debrecen.} 52 (1998), 144-159




\bibitem{Littlewood}
J.E. Littlewood.
\newblock On the function $1/\zeta(1+it)$.
\newblock {\em Proc. London Math. Soc}, 27:358--372, 1928.



\bibitem{Mazur}
B.~Mazur.
\newblock Finding meaning in error terms.
\newblock {\em Bull. Amer. Math. Soc. (N.S.)}, 45(2):185--228, 2008.

\bibitem{Motohashi}
Y.~Motohashi.
\newblock {\em Spectral theory of the {R}iemann zeta-function}, volume 127 of
  {\em Cambridge Tracts in Mathematics}.
\newblock Cambridge University Press, Cambridge, 1997.

\bibitem{Ramachandra2}
K.~Ramachandra.
\newblock A remark on {$\zeta(1+it)$}.
\newblock {\em Hardy-Ramanujan J.}, 10:2--8, 1987.

\bibitem{Ramachandra}
K.~Ramachandra.
\newblock {\em On the mean-value and omega-theorems for the {R}iemann
  zeta-function}, volume~85 of {\em Tata Institute of Fundamental Research
  Lectures on Mathematics and Physics}.
\newblock Published for the Tata Institute of Fundamental Research, Bombay,
  1995.

\bibitem{Rankin}
R.~A. Rankin.
\newblock Contributions to the theory of Ramanujan's function $\tau(n)$ and similar arithmetic functions. I. The zeros of the function $\sum_{n=1}^\infty \tau(n)/n^s$ on the line $\Re(s)=13/2$. II. The order fo the Fourier coefficients of integral modular forms.
\newblock {\em Proc. Cambridge Philos. Soc.}, 35:351--372, 1939.

\bibitem{Selberg2}
A.~Selberg.
\newblock Bemerkungen \"uber eine Dirchletsche Reihe, die mit der Theorie der Modulformen nahe verbunden ist.
\newblock {\em Arch. Math. Naturvid.}, 43:47--50, 1940.

\bibitem{Selberg}
A.~Selberg.
\newblock On the zeros of {R}iemann's zeta-function.
\newblock {\em Skr. Norske Vid. Akad. Oslo I.}, 1942(10):59, 1942.

\bibitem{Serre}
J-P.~Serre.
\newblock {\em Abelian {$l$}-adic representations and elliptic curves}.
\newblock McGill University lecture notes written with the collaboration of
  Willem Kuyk and John Labute. W. A. Benjamin, Inc., New York-Amsterdam, 1968.

\bibitem{Steuding}
J.~Steuding.
\newblock {\em Value-distribution of {$L$}-functions}, volume 1877 of {\em
  Lecture Notes in Mathematics}.
\newblock Springer, Berlin, 2007.

\bibitem{Taylor2}
R.~Taylor.
\newblock Automorphy for some {$l$}-adic lifts of automorphic mod {$l$}
  {G}alois representations. {II}.
\newblock {\em Publ. Math. Inst. Hautes \'Etudes Sci.}, (108):183--239, 2008.

\bibitem{Tenenbaum}
G.~Tenenbaum.
\newblock Remarques sur les valeurs moyennes de fonctions multiplicatives.
\newblock {\em Enseign. Math. (2)}, 53(1-2):155--178, 2007.

\bibitem{Titchmarsh}
E.~C. Titchmarsh.
\newblock {\em The theory of the {R}iemann zeta-function}.
\newblock The Clarendon Press Oxford University Press, New York, second
  edition, 1986.
\newblock Edited and with a preface by D. R. Heath-Brown.

\bibitem{Vinogradov}
I.~M. Vinogradov.
\newblock A new estimate of the function {$\zeta (1+it)$}.
\newblock {\em Izv. Akad. Nauk SSSR. Ser. Mat.}, 22:161--164, 1958.

\bibitem{Wiles}
A.~Wiles.
\newblock Modular elliptic curves and {F}ermat's last theorem.
\newblock {\em Ann. of Math. (2)}, 141(3):443--551, 1995.
\bibitem{Voronin1}
S.~M. Voronin.
\newblock The distribution of the nonzero values of the {R}iemann {$\zeta
  $}-function.
\newblock {\em Trudy Mat. Inst. Steklov.}, 128:131--150, 260, 1972.
\newblock Collection of articles dedicated to Academician Ivan Matveevi{\v{c}}
  Vinogradov on his eightieth birthday. II.

\bibitem{Voronin2}
S.~M. Voronin.
\newblock A theorem on the distribution of values of the {R}iemann
  zeta-function.
\newblock {\em Dokl. Akad. Nauk SSSR}, 221(4):771, 1975.

\bibitem{Voronin0}
S.~M. Voronin.
\newblock A theorem on the ``universality'' of the {R}iemann zeta-function.
\newblock {\em Izv. Akad. Nauk SSSR Ser. Mat.}, 39(3):475--486, 703, 1975.

\bibitem{Wu}
J.~Wu.
\newblock Power sums of {H}ecke eigenvalues and application.
\newblock {\em Acta Arith.}, 137(4):333--344, 2009.

\end{thebibliography}
\end{document}